\documentclass{amsart}
\usepackage{amsfonts}
\usepackage{amsmath}
\usepackage{amssymb}
\usepackage{graphicx}

\newtheorem{theorem}{Theorem}
\theoremstyle{plain}

\newtheorem{case}{Case}

\newtheorem{definition}{Definition}

\newtheorem{lemma}{Lemma}

\newtheorem{remark}{Remark}

\numberwithin{equation}{section}

\begin{document}
\title[Pseudo-Differential Operators and Bessel Potentials]{Non-Archimedean
Pseudo-Differential Operators with Bessel potentials}

\begin{abstract}
In this article, we study a class of non-archimedean pseudo-differential
operators associated via Fourier transform to the Bessel potentials. These
operators (which we will denote as $J^{\alpha },$ $\alpha >n$) are of the
form
\begin{equation*}
(J^{\alpha }\varphi )(x)=\mathcal{F}_{\xi \rightarrow x}^{-1}\left[
(\max
\{1,||\xi ||_{p}\})^{-\alpha }\widehat{\varphi }(\xi )\right] ,\text{ }%
\varphi \in \mathcal{D}(%
\mathbb{Q}
_{p}^{n}),\text{ }x\in
\mathbb{Q}
_{p}^{n}.
\end{equation*}%
We show that the fundamental solution $Z(x,t)$ of the $p-$adic heat equation
naturally associated to these operators satisfies $Z(x,t)\leq 0,x\in
\mathbb{Q}
_{p}^{n},t>0.$ So this equation describes the cooling (or loss of heat) in a
given region over time.

Unlike the archimedean classical theory, although the operator symbol $%
-J^{\alpha }$ is not a function negative definite, we show that the operator
$-J^{\alpha }$ satisfies the positive maximum principle on $C_{0}(%
\mathbb{Q}
_{p}^{n})$. Moreover, we will show that the closure $\overline{-J^{\alpha }}$
of the operator $-J^{\alpha }$ is single-valued and generates a strongly
continuous, positive, contraction semigroup $\left\{ T(t)\right\} $ on $%
C_{0}(%
\mathbb{Q}
_{p}^{n})$.

On the other hand, we will show that the operator $-J^{\alpha }$ is $m-$%
dissipative and is the infinitesimal generator of a $C_{0}-$semigroup of
contractions $T(t),$ $t\geq 0,$ on $L^{2}(%
\mathbb{Q}
_{p}^{n}).$ The latter will allow us to show that for $f\in
L^{1}([0,T):L^{2}(%
\mathbb{Q}
_{p}^{n}))$, the function
\begin{equation*}
u(t)=T(t)u_{0}+\int\nolimits_{0}^{t}T(t-s)f(s)ds,\text{ \ \ }0\leq
t\leq T,
\end{equation*}%
is the mild solution of the initial value problem
\begin{equation*}
\left\{
\begin{array}{ll}
\frac{\partial u}{\partial t}(x,t)=-J^{\alpha }u(x,t)+f(t) & t>0\text{,\ }%
x\in
\mathbb{Q}
_{p}^{n} \\
&  \\
u(x,0)=u_{0}\in L^{2}(%
\mathbb{Q}
_{p}^{n})\text{.} &
\end{array}%
\right.
\end{equation*}
\end{abstract}

\author{Ismael Guti\'{e}rrez Garc\'{\i}a}
\email{isgutier@uninorte.edu.co}
\author{Anselmo Torresblanca-Badillo}
\email{atorresblanca@uninorte.edu.co}
\address{Universidad del Norte, Departamento de Matem\'{a}ticas y Est\'{a}%
distica, Km. 5 V\'{\i}a Puerto Colombia. Barranquilla, Colombia. }
\maketitle

\section{\protect\bigskip Introduction}

In this article, we study a class of non-archimedean pseudo-differential
operators associated via Fourier transform to the Bessel potentials. If $%
f\in \mathcal{D}^{\prime }(%
\mathbb{Q}
_{p}^{n})$ (here $%
\mathbb{Q}
_{p}^{n}$ denotes the p-adic numbers and $\mathcal{D}^{\prime }(%
\mathbb{Q}
_{p}^{n})$ is called the space of distributions in $%
\mathbb{Q}
_{p}^{n}$), $\alpha \in
\mathbb{C}
$ we define the $n-$dimensional $p-$adic Bessel potential of order $\alpha $
of $f$ by
\begin{equation}
(J^{\alpha }f)^{\wedge }=\left( \max \left[ 1,||x||_{p}\right] \right)
^{-\alpha }\widehat{f}.  \label{equation_1}
\end{equation}%
Suppose $\alpha \in
\mathbb{C}
$ with ${Re}(\alpha )>0.$ Defining on $%
\mathbb{Q}
_{p}^{n}$ and with values in $%
\mathbb{R}
_{+}:=\left\{ r\in
\mathbb{R}
:r\geq 0\right\} $\ the function $K_{\alpha }$ as follows%
\begin{equation*}
K_{\alpha }(x)=\left\{
\begin{array}{c}
\frac{1-p^{-\alpha }}{1-p^{\alpha -n}}\left( ||x||_{p}^{\alpha -n}-p^{\alpha
-n}\right) \Omega (||x||_{p})\text{ if }\alpha \neq n \\
\\
(1-p^{-n})\log _{p}(\frac{p}{||x||_{p}})\Omega (||x||_{p})\text{ \ \ \ \ \ \
\ if }\alpha =n%
\end{array}%
\right.
\end{equation*}%
we have that $K_{\alpha }\in L^{1}(%
\mathbb{Q}
_{p}^{n})$ and $\widehat{K}_{\alpha }(\xi )=(\max \{1,||\xi
||_{p}\})^{-\alpha },$ see Remark \ref{Lemma 5.2 p-138-139}.

For our purposes, in this article, we will consider $\alpha \in
\mathbb{R}
$ with $\alpha >n.$ The condition $\alpha >n$ is completely necessary to
obtain the inequality (\ref{negative}), which will be crucial in this
article.

For $\varphi \in \mathcal{D}(%
\mathbb{Q}
_{p}^{n})$ and taking inverse Fourier transform on both sides of (\ref%
{equation_1}), we have that
\begin{equation*}
(J^{\alpha }\varphi )(x)=\mathcal{F}_{\xi \rightarrow x}^{-1}\left[
\widehat{K}_{\alpha }(\xi )\widehat{\varphi }(\xi )\right]
=\mathcal{F}
_{\xi \rightarrow x}^{-1}\left[ (\max \{1,||\xi ||_{p}\})^{-\alpha }\widehat{%
\varphi }(\xi )\right] ,\text{ }x\in
\mathbb{Q}
_{p}^{n},
\end{equation*}%
is a pseudo-differential operator with symbol $\widehat{K}_{\alpha }(\xi
)=(\max \{1,||\xi ||_{p}\})^{-\alpha }.$

In this paper we study the fundamental solution $($denoted by $%
Z(x,t):=Z_{t}(x),$ $x\in
\mathbb{Q}
_{p}^{n},$ $t>0$ $)$ of the $p-$adic pseudodifferential equations of the
form
\begin{equation}
\left\{
\begin{array}{ll}
\frac{\partial u}{\partial t}(x,t)=J^{\alpha }u(x,t)\text{,} & t\in \left[
0,\infty \right) \text{,\ }x\in
\mathbb{Q}
_{p}^{n} \\
&  \\
u(x,0)=u_{0}(x)\in \mathcal{D}(%
\mathbb{Q}
_{p}^{n})\text{,} &
\end{array}%
\right.  \label{eq_2}
\end{equation}%
which is the $p$-adic counterparts of the archimedean heat equations.

Unlike the fundamental solution studied at \cite{Casas-Zuniga}, \cite{Ch-Z-1}%
, \cite{Gu-To-1}, \cite{Gu-To-2}, \cite{Khrennikov-Kochubei}, \cite%
{Kochubei-1991}, \cite{Kochubei-2001}, \cite{R-Zu}, \cite{To-Z-2}, \cite%
{To-Z}, \cite{Zu}, \cite{Z-G}, \cite{Z1}, et al., we obtain that $Z(x,t)\leq
0,$ $\int_{%
\mathbb{Q}
_{p}^{n}}Z(x,t)d^{n}x=e^{-t},$ $x\in
\mathbb{Q}
_{p}^{n},$ $t>0,$ among other properties, see Theorem \ref{Theorem_Z(x,t)}.
Since the heat kernel contains a large amount of redundant information, in
our case, the $p-$adic heat equation describes the cooling (or loss of heat)
in a given region over time.\ \ \ \

The connections between pseudodifferential operators whose symbol is a
negative definite function and that satisfies the positive maximum principle
have been studied intensively in the archimedean setting, since, a
sufficient and necessary condition for that a pseudodifferential operator
satisfies the positive maximum principle is that its symbol be a negative
definite function, see e.g. \cite{Courrege}, \cite{Hoh-Libro}-\cite%
{Jacob-Schilling}, \cite{Schilling}, et al.\ In our case, the
pseudodifferential operator $-J^{\alpha }$ satisfies the positive maximum
principle on $C_{0}(%
\mathbb{Q}
_{p}^{n})$, however, its symbol $(\max \{1,||\xi ||_{p}\})^{-\alpha }$ is
not a negative definite function, see Theorem \ref{maximum_principle} and
Remark \ref{def_negative_definite}, respectively.

On the other hand, the study of the $m$-dissipative operators self-adjoint
on the Hilbert spaces is of great importance, since these are exactly the
generators of contraction semigroups and $C_{0}$-semigroups, see \cite{C-H},
\cite{Pazy}. Motivated by it, we are interested in knowing if the
pseudo-differential operator $-J^{\alpha }$ is an $m$-dissipative operators
and self-adjoint on $L^{2}(%
\mathbb{Q}
_{p}^{n}).$

The article is organized as follows: In Section \ref{Fourier Analysis}, we
will collect some basic results on the $p$-adic analysis and fix the
notation that we will use through the article. In Section \ref{Section_2},
we study a class of non-archimedean pseudo-differential operators associated
via Fourier transform to the Bessel potentials, these operators we denote by
$J^{\alpha },$ $\alpha \in
\mathbb{C}
$. For our purposes, we will consider the case when $\alpha >n.$ In
addition, we will study certain properties corresponding to the fundamental
solution $Z(x,t),$ $x\in
\mathbb{Q}
_{p}^{n},$ $t>0$ of the $p-$adic heat equation naturally associated to these
operators$.$ In Section \ref{Section_3}, we will show that the operator $%
-J^{\alpha }$ satisfies the positive maximum principle on $C_{0}(%
\mathbb{Q}
_{p}^{n})$. Moreover, as for all $\lambda >0$ we have that $Ran(\lambda
+J^{\alpha })$ is dense in $C_{0}(%
\mathbb{Q}
_{p}^{n})$, we have that the closure $\overline{-J^{\alpha }}$ of the
operator $-J^{\alpha }$ on $C_{0}(%
\mathbb{Q}
_{p}^{n})$ is single-valued and generates a strongly continuous, positive,
contraction semigroup $\left\{ T(t)\right\} $ on $C_{0}(%
\mathbb{Q}
_{p}^{n}),$ see Theorem \ref{the_closure}. In the section \ref{Section_4},
we will show that the operator $-J^{\alpha }:$ $L^{2}(%
\mathbb{Q}
_{p}^{n})\rightarrow L^{2}(%
\mathbb{Q}
_{p}^{n})$ is $m-$dissipative and self-adjoint, see Theorem \ref%
{m-dissipative} and Lemma \ref{self-adjoint}, respectively. We can get that
the linear operator $-J^{\alpha }$ is the infinitesimal generator of a $%
C_{0}-$semigroup of contraction $T(t)$, $t\geq 0,$ on $L^{2}(%
\mathbb{Q}
_{p}^{n}).$ Moreover, when considering the problem the inhomogeneous initial
problem%
\begin{equation*}
\left\{
\begin{array}{ll}
\frac{\partial u}{\partial t}(x,t)=-J^{\alpha }u(x,t)+f(t) & t>0\text{,\ }%
x\in
\mathbb{Q}
_{p}^{n} \\
&  \\
u(x,0)=u_{0}\in L^{2}(%
\mathbb{Q}
_{p}^{n})\text{,} &
\end{array}%
\right.
\end{equation*}%
with $f:[0,T]\rightarrow L^{2}(%
\mathbb{Q}
_{p}^{n}),$ $T>0,$ we have that the function
\begin{equation*}
u(t)=T(t)u_{0}+\int\nolimits_{0}^{t}T(t-s)f(s)ds,\text{ \ \ }0\leq
t\leq T,
\end{equation*}%
is the mild solution of the initial value problem on $[0,T],$ see Remark \ref%
{C_0 semigroups}. \ \ \

\section{\label{Fourier Analysis}Fourier Analysis on $%
\mathbb{Q}
_{p}^{n}$: Essential Ideas}

\subsection{The field of $p$-adic numbers}

Along this article $p$ will denote a prime number. The field of $p-$adic
numbers $%
\mathbb{Q}
_{p}$ is defined as the completion of the field of rational numbers $\mathbb{%
Q}$ with respect to the $p-$adic norm $|\cdot |_{p}$, which is defined as
\begin{equation*}
\left\vert x\right\vert _{p}=\left\{
\begin{array}{lll}
0\text{,} & \text{if} & x=0 \\
&  &  \\
p^{-\gamma }\text{,} & \text{if} & x=p^{\gamma }\frac{a}{b}\text{,}%
\end{array}%
\right.
\end{equation*}%
where $a$ and $b$ are integers coprime with $p$. The integer $\gamma
:=ord(x) $, with $ord(0):=+\infty $, is called the\textit{\ }$p-$\textit{%
adic order of} $x$.

Any $p-$adic number $x\neq 0$ has a unique expansion of the form
\begin{equation*}
x=p^{ord(x)}\sum_{j=0}^{\infty }x_{j}p^{j},
\end{equation*}%
where $x_{j}\in \{0,1,2,\dots ,p-1\}$ and $x_{0}\neq 0$. By using this
expansion, we define \textit{the fractional part of }$x\in \mathbb{Q}_{p}$,
denoted $\{x\}_{p}$, as the rational number
\begin{equation*}
\left\{ x\right\} _{p}=\left\{
\begin{array}{lll}
0\text{,} & \text{if} & x=0\text{ or }ord(x)\geq 0 \\
&  &  \\
p^{ord(x)}\sum_{j=0}^{-ord_{p}(x)-1}x_{j}p^{j}\text{,} & \text{if} &
ord(x)<0.%
\end{array}%
\right.
\end{equation*}%
We extend the $p-$adic norm to $%
\mathbb{Q}
_{p}^{n}$ by taking
\begin{equation*}
||x||_{p}:=\max_{1\leq i\leq n}|x_{i}|_{p},\text{ for }x=(x_{1},\dots
,x_{n})\in
\mathbb{Q}
_{p}^{n}.
\end{equation*}%
For $r\in \mathbb{Z}$, denote by $B_{r}^{n}(a)=\{x\in
\mathbb{Q}
_{p}^{n};||x-a||_{p}\leq p^{r}\}$ \textit{the ball of radius }$p^{r}$
\textit{with center at} $a=(a_{1},\dots ,a_{n})\in
\mathbb{Q}
_{p}^{n}$, and take $B_{r}^{n}(0)=:B_{r}^{n}$. Note that $%
B_{r}^{n}(a)=B_{r}(a_{1})\times \cdots \times B_{r}(a_{n})$, where $%
B_{r}(a_{i}):=\{x\in
\mathbb{Q}
_{p};|x_{i}-a_{i}|_{p}\leq p^{r}\}$ is the one-dimensional ball of radius $%
p^{r}$ with center at $a_{i}\in
\mathbb{Q}
_{p}$. The ball $B_{0}^{n}$ equals the product of $n$ copies of $B_{0}=%
\mathbb{Z}_{p}$, \textit{the ring of }$p-$\textit{adic integers of }$%
\mathbb{Q}
_{p}$. We also denote by $S_{r}^{n}(a)=\{x\in \mathbb{Q}%
_{p}^{n};||x-a||_{p}=p^{r}\}$ \textit{the sphere of radius }$p^{r}$ \textit{%
with center at} $a=(a_{1},\dots ,a_{n})\in
\mathbb{Q}
_{p}^{n}$, and take $S_{r}^{n}(0)=:S_{r}^{n}$. The balls and spheres are
both open and closed subsets in $%
\mathbb{Q}
_{p}^{n}$.

As a topological space $\left(
\mathbb{Q}
_{p}^{n},||\cdot ||_{p}\right) $ is totally disconnected, i.e. the only
connected \ subsets of $%
\mathbb{Q}
_{p}^{n}$ are the empty set and the points. A subset of $%
\mathbb{Q}
_{p}^{n}$ is compact if and only if it is closed and bounded in $%
\mathbb{Q}
_{p}^{n}$, see e.g. \cite[Section 1.3]{V-V-Z}, or \cite[Section 1.8]%
{Albeverio et al}. The balls and spheres are compact subsets. Thus $\left(
\mathbb{Q}
_{p}^{n},||\cdot ||_{p}\right) $ is a locally compact topological space.

We will use $\Omega \left( p^{-r}||x-a||_{p}\right) $ to denote the
characteristic function of the ball $B_{r}^{n}(a)$. We will use the notation
$1_{A}$ for the characteristic function of a set $A$. Along the article $%
d^{n}x$ will denote a Haar measure on $%
\mathbb{Q}
_{p}^{n}$ normalized so that $\int_{%
\mathbb{Z}
_{p}^{n}}d^{n}x=1.$

\subsection{Some function spaces}

A complex-valued function $\varphi $ defined on $%
\mathbb{Q}
_{p}^{n}$ is called \textit{locally constant} if for any $x\in
\mathbb{Q}
_{p}^{n}$ there exist an integer $l(x)\in \mathbb{Z}$ such that \ \
\begin{equation*}
\varphi (x+x^{\prime })=\varphi (x)\text{ for }x^{\prime }\in B_{l(x)}^{n}.
\end{equation*}%
A function $\varphi :%
\mathbb{Q}
_{p}^{n}\rightarrow \mathbb{C}$ is called a \textit{Bruhat-Schwartz function
(or a test function)} if it is locally constant with compact support. The $%
\mathbb{C}$-vector space of Bruhat-Schwartz functions is denoted by $%
\mathcal{D}(%
\mathbb{Q}
_{p}^{n})=:\mathcal{D}$. Let $\mathcal{D}^{\prime }(%
\mathbb{Q}
_{p}^{n})=:\mathcal{D}^{\prime }$ denote the set of all continuous
functional (distributions) on $\mathcal{D}$. The natural pairing $\mathcal{D}%
^{\prime }(%
\mathbb{Q}
_{p}^{n})\times \mathcal{D}(%
\mathbb{Q}
_{p}^{n})\rightarrow \mathbb{C}$ is denoted as $\left( T,\varphi \right) $
for $T\in \mathcal{D}^{\prime }(%
\mathbb{Q}
_{p}^{n})$ and $\varphi \in \mathcal{D}(%
\mathbb{Q}
_{p}^{n})$, see e.g. \cite[Section 4.4]{Albeverio et al}.

Every $f\in $ $L_{loc}^{1}$ defines a distribution $f\in \mathcal{D}^{\prime
}\left(
\mathbb{Q}
_{p}^{n}\right) $ by the formula
\begin{equation*}
\left( f,\varphi \right) =\int\limits_{%
\mathbb{Q}
_{p}^{n}}f\left( x\right) \varphi \left( x\right) d^{n}x.
\end{equation*}%
Such distributions are called \textit{regular distributions}.

Given $\rho \in \lbrack 0,\infty )$, we denote by $L^{\rho }\left(
\mathbb{Q}
_{p}^{n},d^{n}x\right) =L^{\rho }\left(
\mathbb{Q}
_{p}^{n}\right) :=L^{\rho },$ the $%
\mathbb{C}
-$vector space of all the complex valued functions $g$ satisfying $\int_{%
\mathbb{Q}
_{p}^{n}}\left\vert g\left( x\right) \right\vert ^{\rho }d^{n}x<\infty $, $%
L^{\infty }\allowbreak :=L^{\infty }\left(
\mathbb{Q}
_{p}^{n}\right) =L^{\infty }\left(
\mathbb{Q}
_{p}^{n},d^{n}x\right) $ denotes the $%
\mathbb{C}
-$vector space of all the complex valued functions $g$ such that the
essential supremum of $|g|$ is bounded.

Let denote by $C(%
\mathbb{Q}
_{p}^{n},\mathbb{C})=:C_{\mathbb{C}}$ the $%
\mathbb{C}
-$vector space of all the complex valued functions which are continuous, by $%
C(%
\mathbb{Q}
_{p}^{n},\mathbb{R})=:C_{\mathbb{R}}$ the $\mathbb{R}-$vector space of
continuous functions. Set
\begin{equation*}
C_{0}(%
\mathbb{Q}
_{p}^{n},\mathbb{C}):=\left\{ f:%
\mathbb{Q}
_{p}^{n}\rightarrow
\mathbb{C}
;\text{ }f\text{ is continuous and }\lim_{x\rightarrow \infty
}f(x)=0\right\} ,
\end{equation*}%
where $\lim_{x\rightarrow \infty }f(x)=0$ means that for every $\epsilon >0$
there exists a compact subset $B(\epsilon )$ such that $|f(x)|<\epsilon $
for $x\in
\mathbb{Q}
_{p}^{n}\backslash B(\epsilon ).$ We recall that $(C_{0}(%
\mathbb{Q}
_{p}^{n},\mathbb{C}),||\cdot ||_{L^{\infty }})$ is a Banach space.

\subsection{Fourier transform}

Set $\chi _{p}(y)=\exp (2\pi i\{y\}_{p})$ for $y\in
\mathbb{Q}
_{p}$. The map $\chi _{p}(\cdot )$ is an additive character on $%
\mathbb{Q}
_{p}$, i.e. a continuous map from $\left(
\mathbb{Q}
_{p},+\right) $ into $S$ (the unit circle considered as multiplicative
group) satisfying $\chi _{p}(x_{0}+x_{1})=\chi _{p}(x_{0})\chi _{p}(x_{1})$,
$x_{0},x_{1}\in
\mathbb{Q}
_{p}$. The additive characters of $%
\mathbb{Q}
_{p}$ form an Abelian group which is isomorphic to $\left(
\mathbb{Q}
_{p},+\right) $, the isomorphism is given by $\xi \mapsto \chi _{p}(\xi x)$,
see e.g. \cite[Section 2.3]{Albeverio et al}.

Given $x=(x_{1},\dots ,x_{n}),$ $\xi =(\xi _{1},\dots ,\xi _{n})\in
\mathbb{Q}
_{p}^{n}$, we set $x\cdot \xi :=\sum_{j=1}^{n}x_{j}\xi _{j}$. If $f\in L^{1}$
its Fourier transform is defined by
\begin{equation*}
(\mathcal{F}f)(\xi )=\int_{%
\mathbb{Q}
_{p}^{n}}\chi _{p}(\xi \cdot x)f(x)d^{n}x,\quad \text{for }\xi \in
\mathbb{Q}
_{p}^{n}.
\end{equation*}%
We will also use the notation $\mathcal{F}_{x\rightarrow \xi }f$ and $%
\widehat{f}$\ for the Fourier transform of $f$. The Fourier transform is a
linear isomorphism from $\mathcal{D}(%
\mathbb{Q}
_{p}^{n})$ onto itself satisfying
\begin{equation}
(\mathcal{F}(\mathcal{F}f))(\xi )=f(-\xi ),  \label{FF(f)}
\end{equation}%
for every $f\in \mathcal{D}(%
\mathbb{Q}
_{p}^{n}),$ see e.g. \cite[Section 4.8]{Albeverio et al}. If $f\in L^{2},$
its Fourier transform is defined as
\begin{equation*}
(\mathcal{F}f)(\xi )=\lim_{k\rightarrow \infty }\int_{||x||\leq p^{k}}\chi
_{p}(\xi \cdot x)f(x)d^{n}x,\quad \text{for }\xi \in
\mathbb{Q}
_{p}^{n},
\end{equation*}%
where the limit is taken in $L^{2}.$ We recall that the Fourier transform is
unitary on $L^{2},$ i.e. $||f||_{L^{2}}=||\mathcal{F}f||_{L^{2}}$ for $f\in
L^{2}$ and that (\ref{FF(f)}) is also valid in $L^{2}$, see e.g. \cite[%
Chapter $III$, Section 2]{Taibleson}.

\section{\label{Section_2}Pseudodifferential Operators and Heat Kernel
Associated with Bessel Potentials}

In this section, we study a class of non-archimedean pseudo-differential
operators associated via Fourier transform to the Bessel potentials. We will
also study some aspects associated with the fundamental solution of the heat
equation associated with these operators.

\begin{definition}
\cite[Definition-p. 137]{Taibleson} If $f\in \mathcal{D}^{\prime }(%
\mathbb{Q}
_{p}^{n})$, $\alpha \in
\mathbb{C}
$ we define the $n-$dimensional $p-$adic Bessel potential of order $\alpha $
of $f$ by
\begin{equation}
(J^{\alpha }f)^{\wedge }=\left( \max \left[ 1,||x||_{p}\right] \right)
^{-\alpha }\widehat{f}.  \label{Def_Bessel}
\end{equation}%
We will define the distribution with compact support $G^{\alpha }$ as
\begin{equation}
\widehat{G}^{\alpha }(x)=\left( \max \left[ 1,||x||_{p}\right] \right)
^{-\alpha }.  \label{G_alpha}
\end{equation}
\end{definition}

\begin{remark}
\label{Proposition 5.1-p-173}\cite[Proposition 5.1-p. 137]{Taibleson} For $%
\alpha ,$ $\beta \in
\mathbb{C}
,$ $f\in \mathcal{D}^{\prime }(%
\mathbb{Q}
_{p}^{n}),$ we have that $J^{\alpha }f=G^{\alpha }\ast f\in \mathcal{D}%
^{\prime }(%
\mathbb{Q}
_{p}^{n})$ and $J^{\alpha }(J^{\beta }f)=J^{\alpha +\beta }f.$ The map $%
\varphi \longrightarrow J^{\alpha }\varphi $ is a homeomorphism from $%
\mathcal{D}(%
\mathbb{Q}
_{p}^{n})$ onto $\mathcal{D}(%
\mathbb{Q}
_{p}^{n}).$ Furthermore $J^{\alpha }$ is continuous in $\alpha $ in the
sense that whenever $\{\alpha _{k}\}\rightarrow \alpha $ in $%
\mathbb{C}
$ then $J^{\alpha _{k}}\varphi \rightarrow J^{\alpha }\varphi ,$ when $%
\varphi \in \mathcal{D}(%
\mathbb{Q}
_{p}^{n}).$
\end{remark}

The $n$-dimensional $p-$adic gamma function $\Gamma _{p}^{n}$ is defined as
\begin{equation*}
\Gamma _{p}^{n}(\alpha )=\frac{1-p^{\alpha -n}}{1-p^{-\alpha }}\text{ for }%
\alpha \in
\mathbb{C}
\backslash \{0\}.
\end{equation*}%
Suppose $\alpha \in
\mathbb{C}
$ with $Re(\alpha )>0.$ Define on $%
\mathbb{Q}
_{p}^{n}$ and with values in $%
\mathbb{R}
_{+}$\ the function $K_{\alpha }$ as follows:%
\begin{equation}
K_{\alpha }(x)=\left\{
\begin{array}{c}
\frac{1}{\Gamma _{p}^{n}(\alpha )}\left( ||x||_{p}^{\alpha -n}-p^{\alpha
-n}\right) \Omega (||x||_{p})\text{ if }\alpha \neq n \\
\\
(1-p^{-n})\log _{p}(\frac{p}{||x||_{p}})\Omega (||x||_{p})\text{ \ \ \ \ \ \
\ if }\alpha =n%
\end{array}%
\right.  \label{Def_K_alpha}
\end{equation}

\begin{remark}
\label{Lemma 5.2 p-138-139} $(i)$ Note that $K_{\alpha }(x)$, $x\in
\mathbb{Q}
_{p}^{n},$ is a non-negative radial function. We have that $G^{\alpha
}=K_{\alpha },$ $K_{\alpha }\in L^{1}(%
\mathbb{Q}
_{p}^{n})$ and $\widehat{K}_{\alpha }(\xi )=(\max \{1,||\xi
||_{p}\})^{-\alpha }$, see \cite[Lemma 5.2-p. 138]{Taibleson}. Moreover, can
verify that\ $\int\nolimits_{%
\mathbb{Q}
_{p}^{n}}K_{\alpha }(x)d^{n}x=1,$ if $Re(\alpha )>0,$ see \cite[%
Remarks-p. 138 and (5.5)-p. 139]{Taibleson}.

$(ii)$ Since the function $f(x)=1,$ $x\in
\mathbb{Q}
_{p}^{n},$ is constant we have in particular that $f$ is locally constant.
Moreover, the function $||\cdot ||_{p}$ is also locally constant, see \cite[%
Example 1-p. 79]{V-V-Z}. Therefore the function $(\max \{1,||\xi
||_{p}\})^{-\alpha }$ is locally constant. \ \ \
\end{remark}

For our purposes from now on we consider fix $\alpha >n.$ Therefore,
\begin{equation}
\frac{1-p^{-\alpha }}{1-p^{\alpha -n}}<0.  \label{negative}
\end{equation}
Following the notation given in \cite{Taibleson} and taking into account
Remark \ref{Proposition 5.1-p-173} and Remark \ref{Lemma 5.2 p-138-139}, for
$\varphi \in \mathcal{D}(%
\mathbb{Q}
_{p}^{n})$ we define the pseudo-differential operator $J^{\alpha }$ by \ \ \
\begin{equation}
(J^{\alpha }\varphi )(x):=\mathcal{F}_{\xi \rightarrow x}^{-1}\left[
\widehat{K}_{\alpha }(\xi )\widehat{\varphi }(\xi )\right]
=\mathcal{F}
_{\xi \rightarrow x}^{-1}\left[ (\max \{1,||\xi ||_{p}\})^{-\alpha }\widehat{%
\varphi }(\xi )\right] ,\text{ }x\in
\mathbb{Q}
_{p}^{n},  \label{Def_J}
\end{equation}%
with symbol $\widehat{K}_{\alpha }(\xi )=(\max \{1,||\xi ||_{p}\})^{-\alpha
}.$

Nothe that if $\varphi \in \mathcal{D}(%
\mathbb{Q}
_{p}^{n})$ then $\widehat{\varphi }\in \mathcal{D}(%
\mathbb{Q}
_{p}^{n}),$ see \cite[Lemma 4.8.1]{V-V-Z}. Moreover, $supp(\widehat{K}%
_{\alpha }\widehat{\varphi })=supp(\widehat{\varphi }),$ so that by Remark %
\ref{Lemma 5.2 p-138-139}-$(ii)$ we have that $\widehat{K}_{\alpha }\widehat{%
\varphi }\in \mathcal{D}(%
\mathbb{Q}
_{p}^{n}).$ Therefore the operator $J^{\alpha }:\mathcal{D}(%
\mathbb{Q}
_{p}^{n})\rightarrow \mathcal{D}(%
\mathbb{Q}
_{p}^{n})$ is well defined, and by Remark \ref{Proposition 5.1-p-173} and
Remark \ref{Lemma 5.2 p-138-139}, we have that
\begin{equation*}
(J^{\alpha }\varphi )(x)=(K_{\alpha }\ast \varphi )(x),\text{ for }\varphi
\in \mathcal{D}(%
\mathbb{Q}
_{p}^{n}).
\end{equation*}

\begin{lemma}
\label{int_exp}For $t>0$ we have that
\begin{equation}
\int\nolimits_{%
\mathbb{Q}
_{p}^{n}}e^{-t\widehat{K}_{\alpha }(\xi )}\leq e^{-t}-1,  \label{cota e-t}
\end{equation}%
i.e. $e^{-t\widehat{K}_{\alpha }(\xi )}=e^{-t(\max \{1,||\xi
||_{p}\})^{-\alpha }}\in L^{1}(%
\mathbb{Q}
_{p}^{n}).$
\end{lemma}

\begin{proof}
Since $e^{-tp^{-j\alpha }}\leq 1$ for $j\geq 1,$ we have that
\begin{eqnarray*}
\int\nolimits_{%
\mathbb{Q}
_{p}^{n}}e^{-t\widehat{K}_{\alpha }(\xi )}d^{n}\xi &=&\int\nolimits_{%
\mathbb{Q}
_{p}^{n}}e^{-t(\max \{1,||\xi ||_{p}\})^{-\alpha }}d^{n}\xi \\
&=&e^{-t}+(1-p^{-n})\sum_{j=1}^{\infty }e^{-tp^{-j\alpha }}p^{nj} \\
&\leq &e^{-t}+\sum_{j=1}^{\infty }(p^{nj}-p^{n(j-1)})=e^{-t}-1.
\end{eqnarray*}
\end{proof}

The proof of the following Lemma is similar to the one given in \cite[%
Proposition 1]{To-Z}.\ \ \

\begin{lemma}
Consider the Cauchy problem%
\begin{equation}
\left\{
\begin{array}{ll}
\frac{\partial u}{\partial t}(x,t)=J^{\alpha }u(x,t)\text{,} & t\in \left[
0,\infty \right) \text{,\ }x\in
\mathbb{Q}
_{p}^{n} \\
&  \\
u(x,0)=u_{0}(x)\in \mathcal{D}(%
\mathbb{Q}
_{p}^{n})\text{.} &
\end{array}%
\right.  \label{Cauchy_Problem}
\end{equation}%
Then%
\begin{equation*}
u(x,t)=\int_{\mathbf{%
\mathbb{Q}
}_{p}^{n}}\chi _{p}\left( -\xi \cdot x\right) e^{-t(\max \{1,||\xi
||_{p}\})^{-\alpha }}\widehat{u_{0}}(\xi )d^{n}\xi
\end{equation*}%
is a classical solution of (\ref{Cauchy_Problem}). In addition, $u(\cdot ,t)$
is a continuous function for any $t\geq 0$.
\end{lemma}

We define the heat Kernel attached to operator $J^{\alpha }$ as
\begin{equation*}
Z(x,t):=\mathcal{F}_{\xi \rightarrow x}^{-1}(e^{-t(\max \{1,||\xi
||_{p}\})^{-\alpha }}).
\end{equation*}%
When considering $Z(x,t)$ as a function of $x$ for $t$ fixed, we will write $%
Z_{t}(x).$ On the other hand, by Lemma \ref{int_exp} and , \cite[Chapter
III-Theorem 1.1-(b)]{Taibleson} we have that $Z_{t}(x)$, $t>0$, is uniformly
continuous. \ \ \

\begin{theorem}
\label{Theorem_Z(x,t)}For any $t>0$, the heat Kernel has the following
properties:

$(i)$ $Z(x,t)=\left\{
\begin{array}{lll}
\sum_{i=0}^{\gamma }p^{in}(e^{-tp^{-i\alpha }}-e^{-tp^{-(i+1)\alpha }})<0 &
\text{if} & \text{ }||x||_{p}=p^{-\gamma },\text{ }\gamma \geq 0, \\
&  &  \\
0, & \text{if} & \text{ }||x||_{p}=p^{\gamma },\text{ }\gamma >1,%
\end{array}%
\right. $

i.e., $Z(x,t)\leq 0$ and $supp(Z(x,t))=%
\mathbb{Z}
_{p}^{n},$ for any $t>0$ and $x\in
\mathbb{Q}
_{p}^{n}.$

$(ii)$ $\int_{%
\mathbb{Q}
_{p}^{n}}Z(x,t)d^{n}x=e^{-t}$ with $x\in
\mathbb{Q}
_{p}^{n}.$

$(iii)$ $Z_{t}(x)\ast Z_{t_{0}}(x)=Z_{t+t_{0}}(x),$ for any $t_{0}>0.$

$(iv)$ $\lim_{t\rightarrow 0^{+}}Z(x,t)=\delta (x).$
\end{theorem}

\begin{proof}
$(i)$ For $x\in
\mathbb{Q}
_{p}^{n}$ and $t>0,$ we have by (\ref{Def_J}) that \ \
\begin{equation}
Z(x,t)=e^{-t}\int_{%
\mathbb{Z}
_{p}^{n}}\chi _{p}(-x\cdot \xi )d^{n}\xi +\int_{%
\mathbb{Q}
_{p}^{n}\backslash
\mathbb{Z}
_{p}^{n}}\chi _{p}(-x\cdot \xi )e^{-t||\xi ||_{p}^{-\alpha }}d^{n}\xi .
\label{exp_Z(x,t)}
\end{equation}%
Consider the following cases:

If $||x||_{p}=p^{\gamma },$ with $\gamma \geq 1,$ then by (\ref{exp_Z(x,t)})
and the $n$-dimensional version of \cite[Example 6-p. 42]{V-V-Z}, we have
that%
\begin{eqnarray*}
Z(x,t) &=&\int_{%
\mathbb{Q}
_{p}^{n}\backslash
\mathbb{Z}
_{p}^{n}}\chi _{p}(-x\cdot \xi )e^{-t||\xi ||_{p}^{-\alpha }}d^{n}\xi \\
&=&\sum_{j=1}^{\infty }e^{-tp^{-j\alpha }}\int_{||p^{j}\xi ||_{p}=1}\chi
_{p}(-x\cdot \xi )d^{n}\xi \\
&=&\sum_{j=1}^{\infty }e^{-tp^{-j\alpha }}p^{nj}\int_{||w||_{p}=1}\chi
_{p}(-p^{-j}x\cdot w)d^{n}w\text{ \ }(\text{taking }w=p^{j}\xi \text{ }).
\end{eqnarray*}%
By using the formula
\begin{equation*}
\int_{||w||_{p}=1}\chi _{p}\left( -p^{-j}x\cdot w\right) d^{n}w=\left\{
\begin{array}{lll}
1-p^{-n}, & \text{if} & \text{ }j\leq -\gamma , \\
-p^{-n}, & \text{if } & \text{\ }j=-\gamma +1, \\
0, & \text{if} & \ j\geq -\gamma +2,%
\end{array}%
\right.
\end{equation*}%
we get that $Z(x,t)=0.$

On the other hand, if $||x||_{p}=p^{-\gamma },$ $\gamma \geq 0,$ then by (%
\ref{exp_Z(x,t)}) and the $n$-dimensional version of \cite[Example 6-p. 42]%
{V-V-Z} we have that
\begin{eqnarray}
Z(x,t) &=&e^{-t}+\sum_{j=1}^{\infty }e^{-tp^{-j\alpha }}\int_{||p^{j}\xi
||_{p}=1}\chi _{p}(-x\cdot \xi )d^{n}\xi  \notag \\
&=&e^{-t}+\sum_{j=1}^{\infty }e^{-tp^{-j\alpha
}}p^{nj}\int_{||w||_{p}=1}\chi _{p}(-p^{-j}x\cdot w)d^{n}w\text{ \ }(\text{%
taking }w=p^{j}\xi \text{ }).  \label{Z(x,t)=}
\end{eqnarray}%
Now, we have that%
\begin{equation}
\int_{||w||_{p}=1}\chi _{p}\left( -p^{-j}x\cdot w\right) d^{n}w=\left\{
\begin{array}{lll}
1-p^{-n}, & \text{if} & \text{ }j\leq \gamma , \\
-p^{-n}, & \text{if } & \text{\ }j=\gamma +1, \\
0, & \text{if} & \ j\geq \gamma +2.%
\end{array}%
\right.  \label{Formule}
\end{equation}

We proceed by induction on $\gamma .$ Note that if $\gamma =0,$ then by (\ref%
{Z(x,t)=}) and (\ref{Formule}) we have that
\begin{equation*}
Z(x,t)=e^{-t}-e^{-tp^{-\alpha }}.
\end{equation*}%
If $\gamma =1,$ then by (\ref{Z(x,t)=}) and (\ref{Formule}) we have that%
\begin{eqnarray*}
Z(x,t) &=&e^{-t}+(p^{n}-1)e^{-tp^{-\alpha }}-p^{n}e^{-tp^{-2\alpha }} \\
&=&(e^{-t}-e^{-tp^{-\alpha }})+p^{n}(e^{-tp^{-\alpha }}-e^{-tp^{-2\alpha }}).
\end{eqnarray*}%
Suppose that
\begin{equation*}
Z(x,t)=\sum_{i=0}^{n}p^{in}(e^{-tp^{-i\alpha }}-e^{-tp^{-(i+1)\alpha }})
\end{equation*}%
is satisfied for $\gamma =n.$

Let's see if the hypothesis is met for $\gamma =n+1.$ By (\ref{Z(x,t)=}) and
(\ref{Formule}) we have that%
\begin{align*}
Z(x,t) &=e^{-t}+(1-p^{-n})e^{-tp^{-\alpha
}}p^{n}+(1-p^{-n})e^{-tp^{-2\alpha }}p^{2n}+\ldots \\
&+(1-p^{-n})e^{-tp^{-\gamma \alpha }}p^{n\alpha
}+(1-p^{-n})e^{-tp^{-(\gamma +1)\alpha }}p^{(\gamma +1)n} \\
&-p^{-n}e^{-tp^{-(\gamma +2)\alpha }}p^{(\gamma +2)n} \\
&=e^{-t}+(p^{n}-1)e^{-tp^{-\alpha }}+(p^{2n}-p^{n})e^{-tp^{-2\alpha
}}+\ldots +(p^{n\alpha }-p^{n(\alpha -1)})e^{-tp^{-\gamma \alpha }} \\
&+(p^{(\gamma +1)n}-p^{n\gamma })e^{-tp^{-(\gamma +1)\alpha }}-p^{(\gamma
+1)n}e^{-tp^{-(\gamma +2)\alpha }}.
\end{align*}%
So by the hypothesis of induction we have that
\begin{align*}
Z(x,t) &=(e^{-t}-e^{-tp^{-\alpha }})+p^{n}(e^{-tp^{-\alpha
}}-e^{-tp^{-2\alpha }})+p^{2n}(e^{-tp^{-2\alpha }}-e^{-tp^{-3\alpha
}})+\ldots \\
&+p^{\gamma n}(e^{-tp^{-\gamma \alpha }}-e^{-tp^{-(\gamma +1)\alpha
}})+p^{(\gamma +1)n}(e^{-tp^{-(\gamma +1)\alpha }}-e^{-tp^{-(\gamma
+2)\alpha }}).
\end{align*}%
Therefore, taking into account that the function $f(x)=e^{-tp^{-x\alpha }}$
is an increasing function in the real variable $x,$ we have that
\begin{equation*}
Z(x,t)\leq 0,\text{ for any }t>0\text{ and }x\in
\mathbb{Q}
_{p}^{n}.
\end{equation*}

$(ii)$ Let $t>0.$ By the definition of $Z(x,t),$ $(i)$ and (\ref{cota e-t}),
we have that $|Z(x,t)|\leq e^{-t}-1.$ Therefore,
\begin{equation}
Z(x,t)\in L^{1}(%
\mathbb{Q}
_{p}^{n}).  \label{Z(x,t)_int}
\end{equation}

Since $\widehat{Z}(x,t)=e^{-t(\max \{1,||\xi ||_{p}\})^{-\alpha }}$ we have
that $\widehat{Z}(0,t)=e^{-t}.$ On the other hand, $\widehat{Z}(\xi
,t)=\int\nolimits_{%
\mathbb{Q}
_{p}^{n}}\chi _{p}(\xi \cdot x)Z(x,t)d^{n}x$ and $\widehat{Z}%
(0,t)=\int\nolimits_{%
\mathbb{Q}
_{p}^{n}}Z(x,t)d^{n}x.$ Therefore, $\int\nolimits_{%
\mathbb{Q}
_{p}^{n}}Z(x,t)d^{n}x=e^{-t}.$

$(iii)$ It is an immediate consequence of the definition of $Z(x,t)$\ and (%
\ref{Z(x,t)_int}).

$(iv)$ For $\varphi \in \mathcal{D}(%
\mathbb{Q}
_{p}^{n})$ we have that
\begin{align*}
\lim_{t\rightarrow 0^{+}}\left\langle Z_{t}(x),\varphi \right\rangle
&=\lim_{t\rightarrow 0^{+}}\left\langle e^{-t(\max \{1,||\xi
||_{p}\})^{-\alpha }},\mathcal{F}_{\xi \rightarrow x}^{-1}(\varphi
)\right\rangle =\left\langle 1,\mathcal{F}_{\xi \rightarrow
x}^{-1}(\varphi
)\right\rangle \\
&=\left\langle \delta ,\varphi \right\rangle .
\end{align*}
\end{proof}

\begin{remark}
\label{convolution_semigroups_nega}$(i)$ By the previous theorem, we have
that the family $(Z_{t})_{t>0}$ it's not a convolution semigroup on $%
\mathbb{Q}
_{p}^{n},$ see e.g. \cite{Berg-Gunnar}.

$(ii)$ By $(ii)$ in the previous theorem we have that the family of
operators
\begin{equation*}
(\Theta (t)f)(x):=\int\nolimits_{%
\mathbb{Q}
_{p}^{n}}Z(x-y,t)f(y)d^{n}y
\end{equation*}%
not preserve the function $f(x)\equiv 1.$ Thus $\Theta (t)$ it's not a
Markov semigroup and moreover the fundamental solution $Z(x,t)$ it's not a
transition density of a Markov process. For more details, the reader can
consult the theory of Markov processes, see e.g. \cite{Dyn}, \cite%
{Kochubei-2001}.\ \

$(iii)$ By Theorem \ref{Theorem_Z(x,t)}-$(ii)$ and Fubini%
\'{}%
s theorem, we have that the classical solution of (\ref{Cauchy_Problem}) can
be written as
\begin{equation*}
u(x,t)=Z_{t}(x)\ast u_{0}(x),\text{ }t\geq 0,\text{ }x\in
\mathbb{Q}
_{p}^{n}.
\end{equation*}
\ \ \
\end{remark}

\section{\label{Section_3}The Positive Maximum Principle and\textit{\ }%
Strongly Continuous, Positive, Contraction Semigroup On $C_{0}(%
\mathbb{Q}
_{p}^{n})$}

In this section, we will show that the operator $-J^{\alpha }$ satisfies the
positive maximum principle on $C_{0}(%
\mathbb{Q}
_{p}^{n})$ and that also the closure $\overline{-J^{\alpha }}$ of the
operator $-J^{\alpha }$ on $C_{0}(%
\mathbb{Q}
_{p}^{n})$ is single-valued and generates a strongly continuous, positive,
contraction semigroup $\left\{ T(t)\right\} $ on $C_{0}(%
\mathbb{Q}
_{p}^{n}).$ For more details, the reader can consult \cite{Berg-Gunnar},
\cite{St-Th}. \ \

\begin{definition}
An operator $(A,Dom(A))$ on $C_{0}(\mathbf{%
\mathbb{Q}
}_{p}^{n})$ is said to satisfy the \textit{positive maximum principle} if
whenever $f\in Dom(A)\subseteq C_{0}(\mathbf{%
\mathbb{Q}
}_{p}^{n},\mathbb{R})$, $x_{0}\in \mathbf{%
\mathbb{Q}
}_{p}^{n}$, and $\sup_{x\in \mathbf{%
\mathbb{Q}
}_{p}^{n}}f(x)=f(x_{0})\geq 0$ we have $Af(x_{0})\leq 0$.
\end{definition}

Let $x_{0}\in
\mathbb{Q}
_{p}^{n}$ such that $\sup_{x\in \mathbf{%
\mathbb{Q}
}_{p}^{n}}\varphi (x)=\varphi (x_{0})\geq 0$ with $\varphi \in \mathcal{D}(%
\mathbb{Q}
_{p}^{n}).$ By Remark \ref{Proposition 5.1-p-173}, Remark \ref{Lemma 5.2
p-138-139} and (\ref{Def_K_alpha}) we have that
\begin{eqnarray}
(J^{\alpha }\varphi )(x_{0}) &=&(K_{\alpha }\ast \varphi )(x_{0})  \notag \\
&=&\frac{1-p^{-\alpha }}{1-p^{\alpha -n}}\int\nolimits_{%
\mathbb{Q}
_{p}^{n}}\left( ||x_{0}||_{p}^{\alpha -n}-p^{\alpha -n}\right) \Omega
(||x_{0}||_{p})\varphi (x_{0}-y)d^{n}y  \label{JA_rep} \\
&=&\frac{1-p^{-\alpha }}{1-p^{\alpha -n}}\int\nolimits_{%
\mathbb{Q}
_{p}^{n}}\left( ||x_{0}-y||_{p}^{\alpha -n}-p^{\alpha -n}\right) \Omega
(||x_{0}-y||_{p})\varphi (y)d^{n}y.  \label{JA_rep2}
\end{eqnarray}%
Consider the following cases:

\begin{case}
$\varphi (x_{0})>0.$ In this case $x_{0}\in supp(\varphi )$ and for all $%
x\in supp(\varphi )$ we have that $\varphi (x)>0.$

If $supp(\varphi )\cap
\mathbb{Z}
_{p}^{n}=\phi ,$ then $||x_{0}||_{p}>p.$ So that by (\ref{JA_rep}) we have
that
\begin{equation*}
(J^{\alpha }\varphi )(x_{0})=0.
\end{equation*}

If $supp(\varphi )\subseteq
\mathbb{Z}
_{p}^{n}$, then $||x_{0}||_{p}=p^{-\beta },$ with $\beta \geq 0.$ For the
case where $||y||_{p}>1$ we have that $||x_{0}-y||_{p}=||y||_{p}$ and
consequently $x_{0}-y\notin
\mathbb{Z}
_{p}^{n}$ and $\varphi (x_{0}-y)=0.$ On the other hand, if $||y||_{p}\leq 1$
then $||x_{0}-y||_{p}\leq 1,$ and in this case $\varphi (x_{0}-y)>0.$ So
that by (\ref{JA_rep}) and (\ref{negative}) we have that\
\begin{equation*}
(J^{\alpha }\varphi )(x_{0})=\frac{1-p^{-\alpha }}{1-p^{\alpha -n}}%
\int\nolimits_{%
\mathbb{Z}
_{p}^{n}}\left( p^{-\beta (\alpha -n)}-p^{\alpha -n}\right) \varphi
(x_{0}-y)d^{n}y\geq 0.
\end{equation*}

If $%
\mathbb{Z}
_{p}^{n}\subseteq supp(\varphi ),$ then there are two possibilities for $%
x_{0}.$ For the case when $||x_{0}||_{p}\leq 1$ we have that $%
||x_{0}||_{p}^{\alpha -n}-p^{\alpha -n}\leq 0$ and $\varphi (x_{0}-y)>0$ for
all $y\in supp(\varphi ).$ So that by (\ref{JA_rep}) and (\ref{negative}),
we have that $(J^{\alpha }\varphi )(x_{0})\geq 0.$ For the case when $%
x_{0}\in supp(\varphi )\backslash
\mathbb{Z}
_{p}^{n}$, by (\ref{JA_rep}) we have that $(J^{\alpha }\varphi )(x_{0})=0.$
\end{case}

\begin{case}
$\varphi (x_{0})=0.$ In this case $x_{0}\notin supp(\varphi )$ and for all $%
x\in supp(\varphi )$ we have that $\varphi (x)<0.$ \

If $supp(\varphi )\subseteq
\mathbb{Z}
_{p}^{n}$, then $||x_{0}||_{p}>1,$ or if $supp(\varphi )\subseteq
\mathbb{Q}
_{p}^{n}\backslash
\mathbb{Z}
_{p}^{n}$ and $x_{0}\notin
\mathbb{Z}
_{p}^{n}$, then by (\ref{JA_rep}) we have that $(J^{\alpha }\varphi
)(x_{0})=0.$

If $%
\mathbb{Z}
_{p}^{n}\subset supp(\varphi )$ and $x_{0}\notin
\mathbb{Z}
_{p}^{n}$ then by (\ref{JA_rep}) we have that $(J^{\alpha }\varphi
)(x_{0})=0,$ and if $%
\mathbb{Z}
_{p}^{n}\subset supp(\varphi )$ and $x_{0}\in
\mathbb{Z}
_{p}^{n}$ then by (\ref{JA_rep2}) we have that $(J^{\alpha }\varphi
)(x_{0})=0.$

If $supp(\varphi )\subset
\mathbb{Q}
_{p}^{n}\backslash
\mathbb{Z}
_{p}^{n}$ and $x_{0}\in
\mathbb{Z}
_{p}^{n}.$ Then, when $y\in
\mathbb{Q}
_{p}^{n}\backslash
\mathbb{Z}
_{p}^{n}$, we have that $||x_{0}-y||_{p}^{\alpha -n}=||y||_{p}^{\alpha -n}$
and $\Omega (||x_{0}-y||_{p})=0.$ Moreover, if $y\in
\mathbb{Z}
_{p}^{n}$ then $\varphi (y)=0.$ Therefore, by (\ref{JA_rep2}) we have that $%
(J^{\alpha }\varphi )(x_{0})=0.$
\end{case}

From all the above we have shown the following theorem.

\begin{theorem}
\label{maximum_principle}The operator
\begin{equation*}
(-J^{\alpha }\varphi )(x)=-\mathcal{F}_{\xi \rightarrow
x}^{-1}\left[ (\max \{1,||\xi ||_{p}\})^{-\alpha }\widehat{\varphi
}(\xi )\right] ,\text{ }x\in
\mathbb{Q}
_{p}^{n},\text{ }\varphi \in \mathcal{D}(%
\mathbb{Q}
_{p}^{n}),
\end{equation*}%
satisfies the \textit{positive maximum principle }on $C_{0}(%
\mathbb{Q}
_{p}^{n}).$
\end{theorem}

\begin{lemma}
\label{densidad}For all $\lambda >0$ we have that $Ran(\lambda +J^{\alpha })$
is dense in $C_{0}(%
\mathbb{Q}
_{p}^{n}).$
\end{lemma}

\begin{proof}
Let $\lambda >0$ and $\varphi \in \mathcal{D}(%
\mathbb{Q}
_{p}^{n}).$ Considering the equation
\begin{equation}
(\lambda +J^{\alpha })u=\varphi ,  \label{equation}
\end{equation}%
we have that $u(\xi )=\mathcal{F}_{\xi \rightarrow x}^{-1}\left( \frac{%
\widehat{\varphi }(\xi )}{\lambda +\widehat{K}_{\alpha }(\xi
)}\right) $ is a solution of the equation (\ref{equation}). \ Since
$\widehat{\varphi }(\xi
)\in \mathcal{D}(%
\mathbb{Q}
_{p}^{n}),$ then by Remark \ref{Lemma 5.2 p-138-139}-$(ii)$ we have that $%
\frac{\widehat{\varphi }(\xi )}{\lambda +\widehat{K}_{\alpha }(\xi )}\in
\mathcal{D}(%
\mathbb{Q}
_{p}^{n}).$ Therefore, $u\in \mathcal{D}(%
\mathbb{Q}
_{p}^{n}).$
\end{proof}

\begin{theorem}
\label{the_closure}The closure $\overline{-J^{\alpha }}$ of the operator $%
-J^{\alpha }$ on $C_{0}(%
\mathbb{Q}
_{p}^{n})$ is single-valued and generates a strongly continuous, positive,
contraction semigroup $\left\{ T(t)\right\} $ on $C_{0}(%
\mathbb{Q}
_{p}^{n})$.
\end{theorem}

\begin{proof}
It follows from Theorem \ref{maximum_principle}, Lemma \ref{densidad} and
\cite[Chapter 4, Theorem 2.2]{St-Th}, taking into account that $\mathcal{D}(%
\mathbb{Q}
_{p}^{n})$ is dense in $C_{0}(%
\mathbb{Q}
_{p}^{n}),$ see e.g. \cite[Proposition 1.3]{Taibleson}.
\end{proof}

\begin{remark}
\label{def_negative_definite}A function $f:\mathbb{%
\mathbb{Q}
}_{p}^{n}\rightarrow
\mathbb{C}
$ is called negative definite, if$\ $%
\begin{equation*}
\sum\nolimits_{i,j=1}^{m}\left( f(x_{i})+\overline{f(x_{j})}%
-f(x_{i}-x_{j})\right) \lambda _{i}\overline{\lambda _{j}}\geq 0
\end{equation*}%
for all $m\in \mathbb{N},$ $x_{1},\ldots ,x_{m}\in $ $\mathbb{%
\mathbb{Q}
}_{p}^{n},$ $\lambda _{1},\ldots ,\lambda _{m}$ $\in $ $\mathbb{C}$.

Nothe that for all $\xi \in
\mathbb{Q}
_{p}^{n}$ we have that $\widehat{K}_{\alpha }(\xi )=(\max \{1,||\xi
||_{p}\})^{-\alpha }\leq \widehat{K}_{\alpha }(0),$ so that by \cite[Chapter
II]{Berg-Gunnar} we have that the function $\widehat{K}_{\alpha }(\xi )$ no
is a function negative definite. Therefore, the operator $-J^{\alpha }$ is a
pseudo-differential operator that satisfies the \textit{positive maximum
principle and }whose symbol is not a negative definite function.
\end{remark}

\section{\label{Section_4}The Pseudo-differential Operator $-J^{\protect%
\alpha }$ on $L^{2}(%
\mathbb{Q}
_{p}^{n})$}

In this section, consider the operator $-J^{\alpha }$ in $L^{2}(%
\mathbb{Q}
_{p}^{n}).$ By \cite[Remarks-p. 138]{Taibleson}, we have that if $f\in L^{2}(%
\mathbb{Q}
_{p}^{n})$ then $(-J^{\alpha }f)\in L^{2}(%
\mathbb{Q}
_{p}^{n})$ and $||-J^{\alpha }f||_{L^{2}(%
\mathbb{Q}
_{p}^{n})}\leq ||f||_{L^{2}(%
\mathbb{Q}
_{p}^{n})}.$ In this case $D(-J^{\alpha })=L^{2}(%
\mathbb{Q}
_{p}^{n}).$ The main objective of this section is to demonstrate  that this
operator is $m-$dissipative, which will be crucial to prove that $-J^{\alpha
}$ is the infinitesimal generator of a $C_{0}-$semigroup of contraction $%
T(t),$ $t\geq 0,$ on $L^{2}(%
\mathbb{Q}
_{p}^{n}).$ For more details, the reader can consult \cite{C-H}, \cite%
{Lunner-Phillips}, \cite{Pazy}, \cite{St-Th}.

\begin{remark}
\label{Graf_closed} The graph $G(-J^{\alpha })$ of $-J^{\alpha }$ is defined
by
\begin{equation*}
G(-J^{\alpha })=\left\{ (u,f)\in L^{2}(%
\mathbb{Q}
_{p}^{n})\times L^{2}(%
\mathbb{Q}
_{p}^{n});\text{ }u\in D(-J^{\alpha })\text{ and }f=-J^{\alpha }u\right\} .
\end{equation*}%
Therefore by \cite[Remark 2.1.6]{C-H} we have that $G(-J^{\alpha })$ is
closed in $L^{2}(%
\mathbb{Q}
_{p}^{n}).$
\end{remark}

\begin{definition}
\cite[Definition 1.2]{Lunner-Phillips} An operator $A$ with domain $D(A)$ in
$L^{2}(%
\mathbb{Q}
_{p}^{n})$ is called dissipative if
\begin{equation*}
re\left\langle Af,f\right\rangle \leq 0,\text{ \ for all }f\in D(A),
\end{equation*}%
where $L^{2}(%
\mathbb{Q}
_{p}^{n})$ is the Hilbert space with the scalar product
\begin{equation*}
\left\langle f,g\right\rangle =\int\nolimits_{%
\mathbb{Q}
_{p}^{n}}f(x)\overline{g}(x)d^{n}x,\text{ \ }f,g\in L^{2}(%
\mathbb{Q}
_{p}^{n}).
\end{equation*}
\end{definition}

\begin{lemma}
\label{Dissipative}The operator $-J^{\alpha }$ is dissipative in $L^{2}(%
\mathbb{Q}
_{p}^{n}).$
\end{lemma}

\begin{proof}
Let fixed $\varphi \in \mathcal{D}(%
\mathbb{Q}
_{p}^{n}).$ Since $\mathcal{D}(%
\mathbb{Q}
_{p}^{n})$ is dense in $L^{2}(%
\mathbb{Q}
_{p}^{n}),$ see \cite[Chapter III]{Taibleson}, and the Parseval-Steklov
equality, see \cite[Section 5.3]{Albeverio et al}, we have that
\begin{eqnarray*}
\left\langle -J^{\alpha }\varphi ,\varphi \right\rangle
&=&\left\langle -\mathcal{F}_{\xi \rightarrow x}^{-1}\left[ (\max
\{1,||\xi ||_{p}\})^{-\alpha }\widehat{\varphi }(\xi )\right]
,\varphi \right\rangle
\\
&=&\left\langle -(\max \{1,||\xi ||_{p}\})^{-\alpha }\widehat{\varphi },%
\widehat{\varphi }\right\rangle  \\
&=&-\int\nolimits_{%
\mathbb{Q}
_{p}^{n}}(\max \{1,||\xi ||_{p}\})^{-\alpha }\left\vert \widehat{\varphi }%
(\xi )\right\vert ^{2}d^{n}\xi \leq 0.
\end{eqnarray*}
\end{proof}

\begin{remark}
By \cite[Chapter 4-Lemma 2.1]{St-Th}, we have that the linear operator $%
-J^{\alpha }$ on $C_{0}(%
\mathbb{Q}
_{p}^{n})$ is dissipative, in the sense that for all $f\in C_{0}(%
\mathbb{Q}
_{p}^{n})$ and $\lambda >0$ we have that $||\lambda f-Af||_{L^{\infty }(%
\mathbb{Q}
_{p}^{n})}\geq \lambda ||f||_{L^{\infty }(%
\mathbb{Q}
_{p}^{n})}.$ This definition is valid for any Banach space with its
corresponding norm, for more details see \cite{C-H}, \cite{Lunner-Phillips}.
\ \ \
\end{remark}

\begin{definition}
\label{m-Dissipative} \cite[Definition 2.2.2]{C-H} An operator $A$ in $L^{2}(%
\mathbb{Q}
_{p}^{n})$ is $m-$dissipative if

$(i)$ $A$ is dissipative;

$(ii)$ for all $\lambda >0$ and all $f\in L^{2}(%
\mathbb{Q}
_{p}^{n}),$ there exists $u\in D(A)$ such that $u-\lambda Au=f.$
\end{definition}

\begin{lemma}
\label{self-adjoint}The operator $-J^{\alpha }$ is self-adjoint, i.e.
\begin{equation*}
\left\langle -J^{\alpha }f,g\right\rangle =\left\langle f,-J^{\alpha
}g\right\rangle ,\text{ for all }f,g\in L^{2}(%
\mathbb{Q}
_{p}^{n}).
\end{equation*}
\end{lemma}

\begin{proof}
For $f,g\in L^{2}(%
\mathbb{Q}
_{p}^{n})$ and the Parseval-Steklov equality, we have that
\begin{align*}
\left\langle -J^{\alpha }f,g\right\rangle &=\left\langle
-\mathcal{F}_{\xi
\rightarrow x}^{-1}\left[ (\max \{1,||\xi ||_{p}\})^{-\alpha }\widehat{f}%
(\xi )\right] ,g\right\rangle \\
&=-\int\nolimits_{%
\mathbb{Q}
_{p}^{n}}(\max \{1,||\xi ||_{p}\})^{-\alpha }\widehat{f}(\xi )\overline{%
\breve{g}}(\xi )d^{n}\xi \\
&=-\int\nolimits_{%
\mathbb{Q}
_{p}^{n}}\overline{(\max \{1,||\xi ||_{p}\})^{-\alpha }}\widehat{f}(\xi )%
\left[ \int\nolimits_{%
\mathbb{Q}
_{p}^{n}}\chi _{p}(x\cdot \xi )\overline{g}(x)d^{n}x\right] d^{n}\xi \\
&=-\int\nolimits_{%
\mathbb{Q}
_{p}^{n}}\widehat{f}(\xi )\overline{(\max \{1,||\xi ||_{p}\})^{-\alpha }}%
\overline{\widehat{g}}(\xi )d^{n}\xi \\
&=\left\langle f,-\mathcal{F}_{\xi \rightarrow x}^{-1}\left[ (\max
\{1,||\xi ||_{p}\})^{-\alpha }\widehat{g}(\xi )\right] \right\rangle
=\left\langle f,-J^{\alpha }g\right\rangle .
\end{align*}
\end{proof}

\begin{theorem}
\label{m-dissipative}The operator $-J^{\alpha }:$ $L^{2}(%
\mathbb{Q}
_{p}^{n})\rightarrow L^{2}(%
\mathbb{Q}
_{p}^{n})$ is $m-$dissipative.
\end{theorem}

\begin{proof}
The result follow from Remark \ref{Graf_closed}, Lemma \ref{Dissipative} and
Lemma \ref{self-adjoint}, by well-known results in the theory of dissipative
operators, see e.g. \cite[Theorem 2.4.5]{C-H}.
\end{proof}

\begin{remark}
\label{generator}The (infinitesimal)\ generator of a semigroups $%
(T(t))_{t\geq 0}$ is the linear operator $L$ defined by
\begin{equation*}
D(L)=\left\{ f\in L^{2}(%
\mathbb{Q}
_{p}^{n});\frac{T(t)x-x}{t}\text{ has a limit in }L^{2}(%
\mathbb{Q}
_{p}^{n})\text{ as }t\rightarrow 0^{+}\right\} ,
\end{equation*}%
and%
\begin{equation*}
Lf=\lim_{t\rightarrow 0^{+}}\frac{T(t)f-f}{t},
\end{equation*}%
for all $f\in D(L).$

The linear operator $-J^{\alpha }$ is the generator of a contraction
semigroups $(T(t))_{t\geq 0}$ in $L^{2}(%
\mathbb{Q}
_{p}^{n}),$ i.e. the family of semigroups $(T(t))_{t\geq 0}$ satisfies:

$(i)$ $||T(t)||_{L^{2}(%
\mathbb{Q}
_{p}^{n})}\leq 1$ for all $t\geq 0;$

$(ii)$ $T(0)=I;$

$(iii)$ $T(t+s)=T(t)T(s)$ for all $s,t\geq 0,$

$(iv)$ for all $f\in L^{2}(%
\mathbb{Q}
_{p}^{n}),$ the function $t\mapsto T(t)f$ belongs to $C([0,\infty ),L^{2}(%
\mathbb{Q}
_{p}^{n}))$.

For more details, the reader can consult \cite[Section 3.4]{C-H}.
\end{remark}

\begin{remark}
\label{C_0 semigroups}By \cite[Theorem 4.3]{Pazy}, \cite[Theorem 4.5-(a)]%
{Pazy}, \cite[Definition 2.1]{Pazy} and Theorem \ref{m-dissipative}, we have
that the linear operator $-J^{\alpha }$ is the infinitesimal generator of a $%
C_{0}-$semigroup of contractions $T(t),$ $t\geq 0,$ on $L^{2}(%
\mathbb{Q}
_{p}^{n}).$

Let's consider the problem the inhomogeneous initial value problem
\begin{equation}
\left\{
\begin{array}{ll}
\frac{\partial u}{\partial t}(x,t)=-J^{\alpha }u(x,t)+f(t) & t>0\text{,\ }%
x\in
\mathbb{Q}
_{p}^{n} \\
&  \\
u(x,0)=u_{0}\in L^{2}(%
\mathbb{Q}
_{p}^{n})\text{,} &
\end{array}%
\right.  \label{Cauchy_problem}
\end{equation}%
where $f:[0,T]\rightarrow L^{2}(%
\mathbb{Q}
_{p}^{n}),$ $T>0.$

Then, for $f\in L^{1}([0,T):L^{2}(%
\mathbb{Q}
_{p}^{n}))$ we have that the function
\begin{equation*}
u(t)=T(t)u_{0}+\int\nolimits_{0}^{t}T(t-s)f(s)ds,\text{ \ \ }0\leq
t\leq T,
\end{equation*}%
is the mild solution of the initial value problem (\ref{Cauchy_problem}) on $%
[0,T].$
\end{remark}

\end{document}